\documentclass[a4paper,10pt]{amsart}
\usepackage{graphicx, amssymb, amsmath, amscd}
\usepackage[arrow,tips,matrix]{xy}
\usepackage{pst-node, hhline, pst-plot}
\usepackage{subfigure}

\newgray{lightgray}{0.9}

\theoremstyle{plain}
\newtheorem{theorem}{Theorem}[section]
\newtheorem{main}{Theorem} 
\newtheorem{proposition}[theorem]{Proposition}
\newtheorem{lemma}[theorem]{Lemma}
\newtheorem{corollary}[theorem]{Corollary}

\theoremstyle{definition}

\newtheorem{definition}[theorem]{Definition}
\newtheorem{remark}[theorem]{Remark}

\numberwithin{table}{section} \numberwithin{figure}{section}
\numberwithin{equation}{section}

\DeclareMathOperator{\id}{id} \DeclareMathOperator{\D}{D}
\DeclareMathOperator{\GL}{GL} \DeclareMathOperator{\PGL}{PGL}
\DeclareMathOperator{\PO}{PO} \DeclareMathOperator{\orthogonal}{O}
\DeclareMathOperator{\SO}{SO} 
\DeclareMathOperator{\res}{res} \DeclareMathOperator{\Iso}{Iso}
 
\DeclareMathOperator{\Irr}{Irr} \DeclareMathOperator{\Vect}{Vect}
\DeclareMathOperator{\Rep}{Rep}

\DeclareMathOperator{\RP}{\mathbf{RP}} \DeclareMathOperator{\pr}{pr}

\newcommand{\complex}[1]{\mathcal{#1}}
\newcommand{\complexK}{\complex{K}}

\newcommand{\tetra}{\mathrm{T}}
\newcommand{\octa}{\mathrm{O}}
\newcommand{\icosa}{\mathrm{I}}
\newcommand{\vect}{\mathrm{vect}}

\newcommand{\lineK}{{\bar{\complexK}}}
\newcommand{\hatK}{{\hat{\complexK}}}

\newcommand{\field}[1]{\mathbb{#1}}

\newcommand{\R}{\field{R}}

\newcommand{\Z}{\field{Z}}

\begin{document}

\title[equivariant vector bundles over real projective plane]{Classification of equivariant vector bundles over real projective plane}
\author{Min Kyu Kim}
\address{Department of Mathematics Education,
Gyeongin National University of Education, San 59-12, Gyesan-dong,
Gyeyang-gu, Incheon, 407-753, Korea}

\email{mkkim@kias.re.kr}



\begin{abstract}
We classify equivariant topological complex vector bundles over real
projective plane under a compact Lie group (not necessarily
effective) action. It is shown that nonequivariant Chern classes and
isotropy representations at (at most) three points are sufficient to
classify equivariant vector bundles over real projective plane
except one case. To do it, we relate the problem to classification
on two-sphere through the covering map because equivariant vector
bundles over two-sphere have been already classified.
\end{abstract}

\maketitle

\section{Introduction} \label{section: introduction}

In the previous paper \cite{Ki}, we have classified equivariant
topological complex vector bundles over $S^2.$ In this paper, we
classify equivariant topological complex vector bundles over real
projective plane. This explains for one of two exceptional cases of
\cite{Ki}.

We will consider only projective linear actions on real projective
plane. We explain for this. Real projective plane $\RP^2$ is defined
as $\R^3 \setminus (0,0,0) \Big/ \sim$ where the equivalence
relation $\sim$ is given by
\begin{equation*}
(x, y, z) \sim (\lambda x, \lambda y, \lambda z)
\end{equation*}
for nonzero real number $\lambda,$ and let
\begin{equation*}
o : S^2 \rightarrow \RP^2
\end{equation*}
be the usual covering map. The projective linear group $\PGL(3, \R)$
is defined as the quotient group $\GL(3, \R) \Big/ Z \big( \GL(3,
\R) \big)$ where $Z \big( \GL(3, \R) \big)$ is the center of $\GL(3,
\R),$ and let
\begin{equation*}
o_{\GL(3, \R)} : \GL(3, \R) \rightarrow \PGL(3, \R)
\end{equation*}
be the qutient map. Let $\PGL(3, \R)$ act usually on $\RP^2.$ For a
compact Lie group $G,$ a topological $G$-action on $\RP^2$ is called
\textit{projective linear} if the action is given by a continuous
homomorphism $\rho : G \rightarrow \PGL(3, \R).$ Since $\PO(3, \R) =
o_{\GL(3, \R)} \big( \orthogonal(3) \big)$ is a maximal compact
subgroup of $\PGL(3, \R),$ we henceforward may assume
\begin{equation}
\label{equation: assumption on rho} \rho(G) \subset \PO(3, \R).
\end{equation}
In Section \ref{section: closed subgroups}, we show that a
topological action on $\RP^2$ by a compact Lie group is conjugate to
a projective linear action.

To state main results, we need introduce some terminologies and
notations. Let a compact Lie group $G$ act projective linearly (not
necessarily effectively) on $\RP^2$ through a homomorphism $\rho : G
\rightarrow \PGL(3, \R)$ to satisfy (\ref{equation: assumption on
rho}). Let $\Vect_G (\RP^2)$ be the set of isomorphism classes of
topological complex $G$-vector bundles over $\RP^2.$ For a bundle
$E$ in $\Vect_G (\RP^2)$ and a point $x$ in $\RP^2,$ denote by $E_x$
the isotropy $G_x$-representation on the fiber at $x.$ Put $H = \ker
\rho,$ i.e. the kernel of the $G$-action on $\RP^2.$ Let $\Irr(H)$
be the set of characters of irreducible complex $H$-representations
which has a $G$-action defined as
\begin{equation*}
(g \cdot \chi) (h) = \chi (g^{-1} h g)
\end{equation*}
for $\chi \in \Irr(H),$ $g \in G,$ $h \in H.$ For $\chi \in
\Irr(H),$ an $H$-representation is called $\chi$-isotypical if its
character is a multiple of $\chi.$ We slightly generalize this
concept. For $\chi \in \Irr(H)$ and a compact Lie group $K$
satisfying $H \lhd K < G$ and $K \cdot \chi = \chi,$ a
$K$-representation $W$ is called $\chi$-\textit{isotypical} if
$\res_H^K W$ is $\chi$-isotypical, and we denote by $\Vect_K (\RP^2,
\chi)$ the set
\begin{equation*}
\Big\{ ~ [E] \in \Vect_K (\RP^2) ~ \big| ~ E_x \text{ is }
\chi\text{-isotypical for each } x \in S^2 ~ \Big\}
\end{equation*}
where $\RP^2$ delivers the restricted $K$-action. As in \cite{Ki},
our classification is reduced to $\Vect_{G_\chi} (\RP^2, \chi)$ for
each $\chi \in \Irr(H).$ Details are found in \cite[Section
2]{CKMS}.

The classification of $\Vect_{G_\chi} (\RP^2, \chi)$ is highly
dependent on the $G_\chi$-action on the base space $\RP^2.$ So, we
would list all possible actions on $\RP^2$ up to conjugacy in terms
of covering action, and then assign an equivariant simplicial (or
CW) complex structure on $\RP^2$ to each action when $\rho(G_\chi)$
is finite. For these, we first introduce some polyhedra in $\R^3.$
Let $P_m$ for $m \ge 3$ be the regular $m$-gon on $xy$-plane in
$\R^3$ whose center is the origin and one of whose vertices is
$(1,0,0).$ And,
\begin{enumerate}
  \item $| \complexK_m |$ is defined as the boundary
  of the convex hull of $P_m,$ $S=(0,0,-1),$
  $N=(0,0,1),$
  \item $| \complexK_\tetra |$ is defined as
  the regular tetrahedron which is the boundary of the
  convex hull of four points
  $(\frac {1} {3}, \frac {1} {3}, \frac {1} {3}),$
  $(-\frac {1} {3}, -\frac {1} {3}, \frac {1} {3}),$
  $(-\frac {1} {3}, \frac {1} {3}, -\frac {1} {3}),$
  $(\frac {1} {3}, -\frac {1} {3}, -\frac {1} {3}),$
  and which is inscribed to $| \complexK_4 |,$
  \item $| \complexK_\icosa |$ is defined as
  a regular icosahedron which has the origin as the center.
\end{enumerate}
With these, denote natural simplicial complex structures on $|
\complexK_m |,$ $| \complexK_\tetra |,$ $| \complexK_\icosa |$ by
$\complexK_m,$ $\complexK_\tetra,$ $\complexK_\icosa,$ respectively.
It is well-known that each closed subgroup of $\SO(3)$ is conjugate
to one of the following subgroups \cite[Theorem 11]{R}:
\begin{enumerate}
  \item $\Z_n$ generated by the rotation $a_n$ through the angle
  $2\pi/n$ around $z$-axis,
  \item $\D_n$ generated by $a_n$ and the rotation $b$ through
  the angle $\pi$ around $x$-axis,
  \item the tetrahedral group $\tetra$ which is the rotation group
  of $| \complexK_\tetra |,$
  \item the octahedral group $\octa$ which is the rotation group
  of $| \complexK_4 |,$
  \item the icosahedral group $\icosa$ which is the rotation group
  of $| \complexK_\icosa |,$
  \item $\SO(2)$ which is the set of rotations around $z$-axis,
  \item $\orthogonal(2)$ which is defined as $\langle \SO(2), b \rangle,$
  \item $\SO(3)$ itself.
\end{enumerate}
Let $\imath: \SO(3) \rightarrow \PO(3) \subset \PGL(3, \R)$ be the
monomorphism defined as
\begin{equation*}
o_{\GL(3, \R)} \circ i_{\GL(3, \R)} ~ \big|_{\SO(3)}
\end{equation*}
where $i_{\GL(3, \R)} : \orthogonal (3) \rightarrow \GL(3, \R)$ is
the usual inclusion. Put
\begin{equation*}
\bar{\rho} : G_\chi \longrightarrow \SO(3), \quad g \longmapsto
\imath^{-1} \big( \rho (g) \big)
\end{equation*}
for $g \in G_\chi$ which is defined by (\ref{equation: assumption on
rho}). And, define the homomorphism
\begin{equation*}
\hat{\rho} : G_\chi \times Z \longrightarrow \orthogonal(3), \quad
(g, g_0^j) \longmapsto \bar{\rho}(g) \cdot g_0^j
\end{equation*}
for $g \in G_\chi,$ $j \in \Z_2$ where $Z$ is the centralizer $\{
\id, -\id \}$ of $\orthogonal(3)$ and we denote $-\id \in Z$ by
$g_0$ to avoid symbolic confusion. In Section \ref{section: closed
subgroups}, it is shown that each compact subgroup of $\PGL(3, \R)$
is conjugate to $\imath (\bar{R})$ for an $\bar{R}$-entry of Table
\ref{table: introduction}. Denote $\imath (\bar{R})$ by $R,$ and put
$\hat{R} = \bar{R} \times Z$ where the notation $\times$ means the
internal direct product of two subgroups in $\orthogonal(3).$
Henceforward, it is assumed that $\rho(G_\chi) = R$ for some
$\bar{R}.$ Then, images of $\rho,$ $\bar{\rho},$ $\hat{\rho}$ are
$R,$ $\bar{R},$ $\hat{R},$ respectively. Note that the case of
$\bar{R} = \Z_2$ is conjugate to $\bar{R} = \D_1.$ So, we exclude
$\D_1$ in Table \ref{table: introduction}. To each finite $\bar{R},$
we assign a simplicial complex $\hatK_{\hat{R}}$ of Table
\ref{table: introduction} where $\complexK_\octa$ is defined in the
below. The reason for our choice of $\hatK_{\hat{R}}$ is explained
in Section \ref{section: relation RP^2 and S^2}. Let
$\orthogonal(3)$ and their subgroups act usually on $\R^3.$ Then,
the underlying space $|\hatK_{\hat{R}}|$ of $\hatK_{\hat{R}}$ is
invariant under the $\hat{R}$-action on $\R^3$ so that
$|\hatK_{\hat{R}}|$ inherits the $\hat{R}$-action from which
$\hatK_{\hat{R}}$ also carries the induced $\hat{R}$-action. And,
$\hatK_{\hat{R}}$ and $|\hatK_{\hat{R}}|$ deliver the $( G_\chi
\times Z )$-action through $\hat{\rho}.$ These actions induce $R$-
and $G_\chi$-actions on $|\hatK_{\hat{R}}|/Z.$ Then, $( G_\chi
\times Z )$- and $G_\chi$- actions on $|\hatK_{\hat{R}}|$ and
$|\hatK_{\hat{R}}|/Z$ are equal to $( G_\chi \times Z )$- and
$G_\chi$- actions on $S^2$ and $\RP^2$ when we regard
$|\hatK_{\hat{R}}|$ and $|\hatK_{\hat{R}}|/Z$ as $S^2$ and $\RP^2,$
respectively. And, we regard $o$ as the orbit map from
$|\hatK_{\hat{R}}|$ to $|\hatK_{\hat{R}}|/Z.$ We can give a natural
equivariant simplicial (or CW) complex structure on
$|\hatK_{\hat{R}}|/Z$ so that $o$ is an equivariant cellular map.
Sometimes, we consider the restricted $G_\chi$-action on $S^2$ and
$|\hatK_{\hat{R}}|.$

\begin{table}[ht]
{\footnotesize
\begin{tabular}{l||c|c|c}
$\bar{R}$                 & $\hatK_{\hat{R}}$    & $\hat{D}_{\hat{R}}$     & $\hat{d}^{-1}$  \\

\hhline{=#=|=|=}

$\Z_n,$ odd $n$           & $\complexK_{2n}$     & $|e^0|$                 & $S$             \\
$\Z_n,$ even $n$          & $\complexK_n$        & $|e^0|$                 & $S$             \\
$\D_n,$ odd $n,$ $n>1$    & $\complexK_{2n}$     & $[v^0, b(e^0)]$         & $S$             \\
$\D_n,$ even $n$          & $\complexK_n$        & $[v^0, b(e^0)]$         & $S$             \\

\hline

$\tetra$                  & $\complexK_\octa$    & $|e^0|$                 & $b(f^{-1})$        \\
$\octa$                   & $\complexK_\octa$    & $[v^0, b(e^0)]$         & $b(f^{-1})$        \\
$\icosa$                  & $\complexK_\icosa$   & $[v^0, b(e^0)]$         & $b(f^{-1})$        \\

\hline

$\SO(2)$                  &                      & $\{ v^0 \}$             & $S$             \\
$\orthogonal(2)$          &                      & $\{ v^0 \}$             & $S$             \\
$\SO(3)$                  &                      & $\{ S \}$               & $S$             \\

\end{tabular}}
\caption{\label{table: introduction} $\hatK_{\hat{R}},$
$\hat{D}_{\hat{R}},$ $\hat{d}^{-1}$ for each $\bar{R}$}
\end{table}

\begin{figure}[ht!]
\begin{center}
\mbox{ \subfigure[$\complexK_4$]{
\begin{pspicture}(-2.5,-3)(2.5,3)\footnotesize

\psline[linewidth=0.5pt](0,2.5)(-2,-0.5)(0,-2.5)(2,0.5)(0,2.5)
\psline[linewidth=0.5pt](-2,-0.5)(1,-0.5)(2,0.5)
\psline[linewidth=0.5pt](0,2.5)(1,-0.5)(0,-2.5)
\psline[linewidth=0.5pt, linestyle=dotted](-2,-0.5)(-1,0.5)(2,0.5)
\psline[linewidth=0.5pt, linestyle=dotted](0,2.5)(-1,0.5)(0,-2.5)

\uput[l](-1.9,-0.5){$v^0$} \uput[ul](1,-0.5){$v^1$}
\uput[ur](1.9,0.5){$v^2$} \uput[ur](-0.9, 0.5){$v^3$}
\uput[d](0,-2.5){$S$} \uput[u](0,2.5){$N$}

\uput[d](-0.4,-0.4){$e^0$} \uput[l](1.7,0.1){$e^1$}
\uput[dl](0.5,0.5){$e^2$} \uput[r](-1.5,0){$e^3$}

\end{pspicture}
}

\subfigure[$\complexK_\octa$]{
\begin{pspicture}(-2.5,-3)(2.5,3)\footnotesize

\psline[linewidth=0.5pt](0,2.5)(-2,-0.5)(0,-2.5)(2,0.5)(0,2.5)
\psline[linewidth=0.5pt](-2,-0.5)(1,-0.5)(2,0.5)
\psline[linewidth=0.5pt](0,2.5)(1,-0.5)(0,-2.5)
\psline[linewidth=0.5pt, linestyle=dotted](-2,-0.5)(-1,0.5)(2,0.5)
\psline[linewidth=0.5pt, linestyle=dotted](0,2.5)(-1,0.5)(0,-2.5)

\uput[l](-1.9,-0.5){$v^0$} \uput[ul](1,-0.5){$v^1$}
\uput[d](0,-2.5){$v^2$}

\uput[d](-0.4,-0.4){$e^0$} \uput[ul](0.6,-1.5){$e^1$}
\uput[dl](-0.9,-1.5){$e^2$}

\uput[ul](0.2,-1.7){$f^{-1}$} \uput[dr](1.1,-0.8){$f^1$}
\uput[ul](0.1,0.3){$f^0$}

\end{pspicture} } }
\end{center}
\caption{\label{figure: introduction} $\complexK_4$ and
$\complexK_\octa$}
\end{figure}

In dealing with equivariant vector bundles over $\RP^2,$ we need to
consider isotropy representations at (at most) three points. We will
express those points as the image of some points in $S^2$ under $o.$
To specify those points, we introduce some more notations. When $m
\ge 3,$ denote by $v^i$ the vertex $\exp \Big( \frac {2 \pi i
\sqrt{-1} ~} m \Big)$ of $\complexK_m,$ and by $e^i$ the edge of
$\complexK_m$ connecting $v^i$ and $v^{i+1}$ for $i \in \Z_m.$ These
notations are illustrated in Figure \ref{figure: introduction}.(a).
When we use the notation $\Z_m$ to denote an index set, it is just
the group $\Z / m\Z$ of integers modulo $m.$ In \cite{Ki},
$\complexK_m$ and its $v^i$'s, $e^i$'s for $m=2$ are also defined.
We would define similar notations for $\complexK_\tetra$ and
$\complexK_\icosa.$ For $\complexK_\tetra$ and $\complexK_\icosa,$
pick two adjacent faces in each case, and call them $f^{-1}$ and
$f^0.$ And, label vertices of $f^{-1}$ as $v^i$ for $i \in \Z_3$ to
satisfy
\begin{enumerate}
  \item $v^0,$ $v^1,$ $v^2$ are arranged
  in the clockwise way around $f^{-1},$
  \item $v^0,$ $v^1$ are contained in $f^{-1} \cap f^0.$
\end{enumerate}
For $i \in \Z_3,$ denote by $e^i$ be the edge connecting $v^i$ and
$v^{i+1},$ and by $f^i$ the face which is adjacent to $f^{-1}$ and
contains the edge $e^i.$ We distinguish the superscripts $-1$ and
$2$ only for $f^i,$ i.e. $f^{-1} \ne f^2$ in contrast to $v^{-1} =
v^2,$ $e^{-1} = e^2.$ Here, we define one more simplicial complex
denoted by $\complexK_\octa$ which is the same simplicial complex
with $\complexK_4$ but has the same convention of notations $v^i,
e^i, f^{-1}, f^i$ with $\complexK_\tetra,$ $\complexK_\icosa.$ Also,
put $|\complexK_\octa|=|\complexK_4|.$ These notations are
illustrated in Figure \ref{figure: introduction}.(b). With these
notations, we explain for $\hat{D}_{\hat{R}}$-entry of Table
\ref{table: introduction}. To each finite $\bar{R},$ we assign a
path $\hat{D}_{\hat{R}}$ (called the (closed)
\textit{one-dimensional fundamental domain}) in $|\hatK_{\hat{R}}|$
which is listed in the third column of Table \ref{table:
introduction} where $b(\sigma)$ is the barycenter of $\sigma$ for
any simplex $\sigma$ and $[x,y]$ is the shortest path in the
underlying space $|\complexK|$ for any simplicial complex
$\complexK$ and two points $x,y$ in $|\complexK|.$ And, let
$\hat{d}^0$ and $\hat{d}^1$ be boundary points of
$\hat{D}_{\hat{R}}$ such that $\hat{d}^0$ is nearer to $v^0$ than
$\hat{d}^1.$ For each finite $\bar{R},$ we define one more point
$\hat{d}^{-1} \in |\hatK_{\hat{R}}|$ which is listed in the fourth
column in Table \ref{table: introduction}. If $\bar{R}$ is
one-dimensional, then denote by $\hat{D}_{\hat{R}}$ the one point
set $\{ v^0 = (1,0,0) \},$ and let $\hat{d}^{-1},$ $\hat{d}^0,$
$\hat{d}^1$ be equal to $S,$ $v^0,$ $v^0,$ respectively. Similarly,
if $\bar{R}$ is three-dimensional, then denote by
$\hat{D}_{\hat{R}}$ the one point set $\{ S \},$ and let
$\hat{d}^{-1},$ $\hat{d}^0,$ $\hat{d}^1$ be all equal to $S.$ So
far, we have defined $\hat{d}^{-1},$ $\hat{d}^0,$ $\hat{d}^1$ for
each $\bar{R}.$ Put $d^i = o ( \hat{d}^i )$ for $i \in I^+$ where
$I= \{ 0, 1 \}$ and $I^+= \{ -1, 0, 1 \}.$ Then, $d^{-1},$ $d^0,$
$d^1$ are wanted points of $\RP^2$ according to $\bar{R},$ and we
will consider the restriction $E|_{\{ d^{-1}, ~ d^0, ~ d^1 \}}$ for
each $E$ in $\Vect_{G_\chi} (\RP^2, \chi).$ We define the semigroup
which will be shown to be equal to the set of all the restrictions.
If $\bar{R}$ is finite, let $C(\hat{d}^i)$ be the shortest path in
$|\hatK_{\hat{R}}|$ connecting $\hat{d}^{-1}$ with $\hat{d}^i$ for
$i \in I.$ Otherwise, let $C(\hat{d}^i)$ be the shortest path in
$S^2$ connecting $\hat{d}^{-1}$ with $\hat{d}^i$ for $i \in I.$ And,
denote images
\begin{equation*}
o (\hat{D}_{\hat{R}}), \quad o \Big( C(\hat{d}^0) \Big), \quad o
\Big( C(\hat{d}^1) \Big)
\end{equation*}
by
\begin{equation*}
D_R, \qquad C(d^0), \qquad C(d^1),
\end{equation*}
respectively.

\begin{definition} \label{definition: A_R}
For $\chi \in \Irr(H),$ assume that $\rho(G_\chi) = R$ for some
$\bar{R}$ of Table \ref{table: introduction}. Let $A_{G_\chi} (
\RP^2, \chi )$ be the semigroup of triples $( W_{d^{-1}}, W_{d^0},
W_{d^1} )$ in $\Rep \Big( (G_\chi)_{d^{-1}} \Big)$ $\times$ $\Rep
\Big( (G_\chi)_{d^0} \Big)$ $\times$ $\Rep \Big( (G_\chi)_{d^1}
\Big)$ satisfying
\begin{enumerate}
\item[i)] $W_{d^{-1}}$ is $\chi$-isotypical,
\item[ii)] $\res_{(G_\chi)_{C(d^i)}}^{(G_\chi)_{d^{-1}}}
      W_{d^{-1}} \cong
      \res_{(G_\chi)_{C(d^i)}}^{(G_\chi)_{d^i}}
      W_{d^i}$ for $i \in I,$
\item[iii)] $\res_{(G_\chi)_{D_R}}^{(G_\chi)_{d^0}}
      W_{d^0} \cong
      \res_{(G_\chi)_{D_R}}^{(G_\chi)_{d^1}}
      W_{d^1},$
\item[iv)] $W_{d^1} \cong ~ ^g W_{d^0}$ if there
      exists $g \in G_\chi$
      such that $g d^0 = d^1$
\end{enumerate}
where $(G_\chi)_T$ is the subgroup of $G_\chi$ fixing all points of
a subset $T \subset \RP^2.$ And, let $p_\vect^\prime :
\Vect_{G_\chi} (\RP^2, \chi) \rightarrow A_{G_\chi} ( \RP^2, \chi )$
be the map defined as $E \mapsto ( E_{d^{-1}}, E_{d^0}, E_{d^1} ).$
\end{definition}

\begin{remark}
It might seem that $A_{G_\chi} ( \RP^2, \chi )$ is defined in a
different way with $A_{G_\chi} ( S^2, \chi )$ of \cite{Ki}. But,
these two definitions are defined in the exactly same way by
\cite[Lemma 3.10.]{Ki}.
\end{remark}

Well-definedness of $p_{\vect}^\prime$ is proved in Proposition
\ref{proposition: well-definedness of A_R}. For notational
simplicity, denote a triple $( W_{d^{-1}}, W_{d^0}, W_{d^1} )$ by
$(W_{d^i})_{i \in I^+}.$ Now, we can state main results. Let $c_1 :
\Vect_{G_\chi} (\RP^2, \chi) \rightarrow H^2 (\RP^2)$ be the map
defined as $[E] \mapsto c_1 (E).$

\begin{main} \label{main: only by isotropy}
Assume that $\rho(G_\chi)$ is equal to $R$ for some $\bar{R}$ of
Table \ref{table: introduction} except $\Z_n$ with odd $n.$ Then,
$p_\vect^\prime$ is an isomorphism.
\end{main}

\begin{main} \label{main: by isotropy and chern}
Assume that $\rho (G_\chi) = R$ for $\bar{R} = \Z_n$ with odd $n.$
Then, the preimage $p_{\vect}^{\prime -1} ( \mathbf{W} )$ has
exactly two elements for each $\mathbf{W}$ in $A_{G_\chi} ( \RP^2,
\chi ).$ And, $[E \oplus E_1] \ne [E \oplus E_2]$ for any bundles
$E,$ $E_1,$ $E_2$ in $\Vect_{G_\chi} (\RP^2, \chi)$ such that
$p_\vect^\prime([E_1]) = p_\vect^\prime([E_2])$ and $[E_1] \ne
[E_2].$ Also, the following hold:
\begin{enumerate}
  \item if $\chi(\id)$ is even, then Chern classes of two elements of
  $p_{\vect}^{\prime -1} ( \mathbf{W} )$ are the same,
  \item if $\chi(\id)$ is odd, then Chern classes of two elements of
  $p_{\vect}^{\prime -1} ( \mathbf{W} )$ are different.
\end{enumerate}
\end{main}

\begin{corollary} \label{corollary: reduction to line bundle}
Assume that $\rho (G_\chi) = R$ for $\bar{R} = \Z_n$ with odd $n.$
Then, we have
\begin{equation*}
\Vect_{G_\chi} (\RP^2, \chi) \cong \Vect_R (\RP^2)
\end{equation*}
as semigroups, and $\Vect_R (\RP^2)$ is generated by line bundles.
Also, $A_R (\RP^2, \id)$ is generated by all the elements with
one-dimensional entries where we simply denote by $\id$ the trivial
character of the trivial group. The number of such elements in $A_R
(\RP^2, \id)$ is equal to $n.$
\end{corollary}

This paper is organized as follows. In Section \ref{section: closed
subgroups}, we list all closed subgroups of $\PGL(3, \R)$ up to
conjugacy, and show that a topological action on $\RP^2$ by a
compact Lie group is conjugate to a projective linear action. Also,
we show that $p_\vect^\prime$ is well-defined. In Section
\ref{section: relation RP^2 and S^2}, we prove Theorem \ref{main:
only by isotropy} through covering action. In Section \ref{section:
proof of Theorem B}, we prove Theorem \ref{main: by isotropy and
chern} and Corollary \ref{corollary: reduction to line bundle}
through equivariant clutching construction.

\section{Compact subgroups of $\PGL (3, \R)$} \label{section: closed subgroups}

In this section, we list all compact subgroups of $\PGL(3, \R)$ up
to conjugacy, and show that a topological action on $\RP^2$ by a
compact Lie group is conjugate to a projective linear action. Also,
we show that $p_\vect^\prime$ is well-defined.

First, we define covering action of the $G_\chi$-action on $\RP^2.$
We call $G_\chi \times Z$ and its action on $S^2$ through
$\hat{\rho}$ the covering group of $G_\chi$ and the covering action
of the $G_\chi$-action on $\RP^2,$ respectively. Here, we note that
two kernels of the $(G_\chi \times Z)$-action on $S^2$ and the
$G_\chi$-action on $\RP^2$ are equal because $\imath$ is injective.
By using covering action, we list all closed subgroups of $\PGL(3,
\R)$ up to conjugacy.

\begin{proposition} \label{proposition: closed subgroups}
Each compact subgroup of $\PGL(3, \R)$ is conjugate to one of $R$
for some $\bar{R}$-entry in Table \ref{table: introduction}.
\end{proposition}

\begin{proof}
Let $K$ be an arbitrary compact subgroup of $\PGL(3, \R),$ and let
$\rho : K \rightarrow \PGL(3, \R)$ be the inclusion. We may assume
that $K \subset \PO (3, \R)$ by (\ref{equation: assumption on rho}).
Since the codomain of $\bar{\rho}$ is equal to $\SO(3),$ the image
of $\bar{\rho}$ is conjugate to one of
\begin{equation*}
\Z_n, \quad \D_n, \quad \tetra, \quad \octa, \quad \icosa, \quad
\SO(2), \quad \orthogonal(2), \quad \SO(3).
\end{equation*}
This shows that the image of $\rho$ is conjugate to one of $R$ for
some $\bar{R}$-entry in Table \ref{table: introduction} because
$\imath$ is injective. So, we obtain a proof.
\end{proof}

Now, we prove that a compact Lie group action on $\RP^2$ is
conjugate to a projective linear action.

\begin{proposition} \label{proposition: conjugate to projective
linear} If a compact Lie group $G$ act topologically on $\RP^2,$
then it is conjugate to a projective linear action.
\end{proposition}

\begin{proof}
By \cite[Theorem I.9.3]{B}, the $G$-action on $\RP^2$ has a covering
$G^\prime$-action on $S^2$ where $G^\prime$ satisfies a short exact
sequence
\begin{equation*}
\langle \id \rangle \rightarrow \Z_2 \longrightarrow G^\prime
\overset{\pr}{\longrightarrow} G \rightarrow \langle \id \rangle
\end{equation*}
for some homomorphism $\pr$ and the $G^\prime$-action on $S^2$
satisfies
\begin{equation*}
o( g^\prime \cdot \hat{x} ) = \pr( g^\prime ) \cdot o( \hat{x} )
\end{equation*}
for $g^\prime \in G^\prime$ and $\hat{x} \in S^2.$ It is well-known
that a topological action on $S^2$ by a compact Lie group is
conjugate to a linear action \cite[Theorem 1.2]{Ko}, \cite{CK}. By
using this, we may assume that the $G^\prime$-action is linear, and
we obtain a proof.
\end{proof}

To prove well-definedness of $p_\vect^\prime,$ we state a basic
lemma.

\begin{lemma}
\label{lemma: elementary lemma on isotropy representation} Let $G$
be a compact Lie group acting topologically on a topological space $X.$ And, let
$E$ be an equivariant vector bundle over $X.$ Then, $^g E_x \cong
E_{g x}$ for each $g \in G$ and $x \in X.$ Also,
\begin{equation*}
\res_{G_x \cap G_{x^\prime}}^{G_x} E_x \cong \res_{G_x \cap
G_{x^\prime}}^{G_{x^\prime}} E_{x^\prime}
\end{equation*}
for any two points $x, x^\prime$ in the same component of the fixed
set $X^{G_x \cap G_{x^\prime}}.$
\end{lemma}

\begin{proposition}
\label{proposition: well-definedness of A_R} $p_\vect^\prime$ is
well-defined.
\end{proposition}

\begin{proof}
To show well-definedness, we should show that $( E_{d^i} )_{i \in
I^+}$ is contained in $A_{G_\chi} ( \RP^2, \chi )$ for any $E$ in
$\Vect_{G_\chi} (\RP^2, \chi).$ By definition of $\Vect_{G_\chi}
(\RP^2, \chi),$ it is easy that $E_{d^{-1}}$ is $\chi$-isotypical
so that $( E_{d^i} )_{i \in I^+}$ satisfies Definition
\ref{definition: A_R}.i). By the first statement of Lemma
\ref{lemma: elementary lemma on isotropy representation}, $( E_{d^i}
)_{i \in I^+}$ satisfies Definition \ref{definition: A_R}.iv). The
remaining are Definition \ref{definition: A_R}.ii), iii). These are
satisfied by the second statement of Lemma \ref{lemma: elementary
lemma on isotropy representation}.
\end{proof}

\section{Relation between equivariant vector bundles
over $\RP^2$ and $S^2.$} \label{section: relation RP^2 and S^2}

In this section, we investigate the relation between semigroups
$\Vect_{G_\chi} (\RP^2, \chi)$ and $\Vect_{G_\chi \times Z} (S^2,
\chi),$ and the relation between $A_{G_\chi} (\RP^2, \chi)$ and
$A_{G_\chi \times Z} (S^2, \chi).$ By these, we can prove Theorem
\ref{main: only by isotropy}.

The first relation is easy. Pullback through the covering map gives
the semigroup homomorphism
\begin{equation}
\label{equation: isomorphism RP^2 and S^2} O : \Vect_{G_\chi}
(\RP^2) \longrightarrow \Vect_{G_\chi \times Z} (S^2), \quad E
\longmapsto o^* E.
\end{equation}
This is an isomorphism because the $Z$-action on $S^2$ is free and
$\RP^2 \cong S^2 / Z,$ see \cite[p. 132]{S}. Since equivariant
vector bundles over $S^2$ have been classified in \cite{Ki}, we can
classify equivariant vector bundles over $\RP^2$ by using this
isomorphism. This is the reason why $\hatK_{\hat{R}}$ of Table
\ref{table: introduction} is defined as the simplicial complex
$\complexK_{\hat{R}}$ of \cite[Table 1.1]{Ki}. In \cite[Section
3]{Ki}, it is shown that the $(G_\chi \times Z)$-action on $\R^3$
through $\hat{\rho}$ preserves $\hatK_{\hat{R}}$ and its underlying
space $|\hatK_{\hat{R}}|.$ Also, it is shown that the $(G_\chi
\times Z)$-orbit of $\hat{D}_{\hat{R}}$ covers the underlying space
$|\hatK_{\hat{R}}^{(1)}|$ of the 1-skeleton $\hatK_{\hat{R}}^{(1)},$
and that $\hat{D}_{\hat{R}}$ is a minimal path satisfying such a
property.

Next, we would show that $A_{G_\chi} ( \RP^2, \chi )$ is isomorphic
to $A_{G_\chi \times Z} ( S^2, \chi ).$ For this, we state an easy
lemma on isotropy subgroups of $G_\chi$ and $G_\chi \times Z.$ For
any $\hat{x} \in S^2$ and $x=o(\hat{x}),$ let
\begin{equation*}
p_1 : G_\chi \times Z \rightarrow G_\chi \qquad \text{ and } \qquad
p_{1, \hat{x}} : (G_\chi \times Z)_{\hat{x}} \rightarrow (G_\chi)_x
\end{equation*}
be the projection to the first components. And, let $p_{1,
\hat{x}}^* : R \Big((G_\chi)_x \Big) \rightarrow R \Big((G_\chi
\times Z)_{\hat{x}} \Big)$ be the isomorphism sending a
$(G_\chi)_x$-representation $W$ to $W$ itself regarded as a $(G_\chi
\times Z)_{\hat{x}}$-representation through $p_{1, \hat{x}}.$

\begin{lemma} \label{lemma: stabilizer correspondence}
For any points $\hat{x}, \hat{y}$ in $S^2$ and their images
$x=o(\hat{x}),$ $y=o(\hat{y})$ in $\RP^2,$ the following hold
\begin{enumerate}
  \item $p_1 \Big( (G_\chi \times Z)_{\hat{x}} \Big) = (G_\chi)_x$ and
  $p_1 \Big|_{(G_\chi \times Z)_{\hat{x}}}$ is injective,
  \item $p_1 \Big( (G_\chi \times Z )_{C(\hat{d}^i)}\Big) =
  (G_\chi)_{C(d^i)}$ and $p_1 \Big|_{(G_\chi \times Z)_{C(\hat{d}^i)}}$ is
  injective for $i \in I,$
  \item $p_1 \Big( (G_\chi \times Z)_{\hat{D}_{\hat{R}}} \Big) =
  (G_\chi)_{D_R}$ and $p_1 \Big|_{(G_\chi \times Z)_{\hat{D}_{\hat{R}}}}$ is
  injective,
  \item $p_{1, \hat{x}}^* (E_x) \cong (o^* E)_{\hat{x}}$ as
  $(G_\chi \times Z)_{\hat{x}}$-representations for each $E$ in the
  set $\Vect_{G_\chi} (\RP^2, \chi).$
\end{enumerate}
\end{lemma}

\begin{proof}
(1) It is easy that $p_1 \Big( (G_\chi \times Z)_{\hat{x}} \Big)
\subset (G_\chi)_x.$ For each $g \in (G_\chi)_x,$ one of $(g, \id)$
and $(g, g_0)$ fixes $\hat{x}$ because $g \cdot \hat{x} = \pm
\hat{x}.$ So, $p_1 \Big( (G_\chi \times Z)_{\hat{x}} \Big) =
(G_\chi)_x$ is obtained. Now, we prove injectivity of $p_1
\Big|_{(G_\chi \times Z)_{\hat{x}}}.$ Pick arbitrary two elements
$(g, g_0^j)$ and $(g^\prime, g_0^{j^\prime})$ of $(G_\chi \times
Z)_{\hat{x}}$ for $j, j^\prime \in \Z_2.$ If $p_1 \big( g, g_0^j
\big) = p_1 \big( g^\prime, g_0^{j^\prime} \big),$ i.e.
$g=g^\prime,$ then $j$ should be equal to $j^\prime.$ This shows
injectivity.

(2) and (3) are obtained by applying (1) to each point $\hat{x}$ in
$C(\hat{d}^i)$ and $\hat{D}_{\hat{R}}.$

(4) is easy by definition of pullback.
\end{proof}

\begin{proposition} \label{proposition: pullback isomorphism for A_R}
Let $P_1 : A_{G_\chi} ( \RP^2, \chi ) \rightarrow A_{G_\chi \times
Z} ( S^2, \chi )$ be the map defined as
\begin{equation*}
\big( W_{d^i} \big)_{i \in I^+} \longmapsto \big( p_{1, \hat{d}^i}^*
W_{d^i} \big)_{i \in I^+}.
\end{equation*}
Then, $P_1$ is bijective.
\end{proposition}

\begin{proof}
For each $\mathbf{W} = ( W_{d^{-1}}, W_{d^0}, W_{d^1} )$ in
$A_{G_\chi} ( \RP^2, \chi ),$ Lemma \ref{lemma: stabilizer
correspondence} says that the element $P_1 ( \mathbf{W} )$ satisfies
Definition \ref{definition: A_R}.i)$\sim$iii). Put $W_{\hat{d}^i} =
p_{1, \hat{d}^i}^* (W_{d^i})$ for $i \in I.$ It is easily shown that
there exists $g$ in $G_\chi \times Z$ such that $g \hat{d}^0 =
\hat{d}^1$ if and only if there exists $g^\prime$ in $G_\chi$ such
that $g^\prime d^0 = d^1$ where $g, g^\prime$ satisfy $p_1 (g) =
g^\prime.$ Then, we can also show that if $^{g^\prime} W_{d^0} \cong
W_{d^1}$ holds, then $^g W_{\hat{d}^0} \cong W_{\hat{d}^1}.$ So,
$P_1 ( \mathbf{W} )$ satisfies Definition \ref{definition: A_R}.iv),
and $P_1 ( \mathbf{W} )$ is contained in $A_{G \times Z} ( S^2, \chi
).$ Similarly, we can also show that $P_1$ has an inverse.
\end{proof}

Now, we can prove Theorem \ref{main: only by isotropy}.

\begin{proof}[Proof of Theorem \ref{main: only by isotropy}]
Observe that $p_\vect^\prime = P_1^{-1} \circ p_\vect \circ O$ where
the map $p_\vect$ on $\Vect_{G_\chi \times Z} (S^2, \chi)$ is
defined in \cite{Ki}. Note that $O,$ $p_\vect,$ $P_1$ are all
isomorphisms by (\ref{equation: isomorphism RP^2 and S^2}),
\cite[Theorem C]{Ki}, Proposition \ref{proposition: pullback
isomorphism for A_R}, respectively. Therefore, $p_\vect^\prime$ is
also an isomorphism.

\begin{equation} \label{equation: p_vect commute}
\SelectTips{cm}{} \xymatrix{ \Vect_{G_\chi} (\RP^2, \chi)
\ar[r]^-{O}
\ar[d]_-{p_\vect^\prime} & \Vect_{G_\chi \times Z} (S^2, \chi) \ar[d]^-{p_\vect} \\
A_{G_\chi} ( \RP^2, \chi )   \ar[r]^-{P_1}
  &  A_{G_\chi \times Z} ( S^2, \chi ) }
\end{equation}
\end{proof}

\section{Proof of Theorem \ref{main: by isotropy and chern}} \label{section: proof of Theorem
B}

In this section, we deal with the case of $\bar{R} = \Z_n$ with odd
$n.$ In this case, we can check that
\begin{equation*}
R_{d^{-1}} \cong \Z_n \qquad \text{and} \qquad R_{d^0} = R_{d^1} =
\langle \id \rangle.
\end{equation*}
Pick an arbitrary $\mathbf{W} = (W_{d^i})_{i \in I^+}$ in
$A_{G_\chi} ( \RP^2, \chi ).$ By \cite[Theorem B]{Ki}, the map
$p_\vect$ of diagram (\ref{equation: p_vect commute}) is surjective
and $p_\vect^{-1} \Big( P_1(\mathbf{W}) \Big)$ consists of two
elements. Also since $O$ and $P_1$ are isomorphic by (\ref{equation:
isomorphism RP^2 and S^2}) and Proposition \ref{proposition:
pullback isomorphism for A_R}, the diagram (\ref{equation: p_vect
commute}) shows that $p_\vect^\prime$ is surjective and
$p_\vect^{\prime -1} \big( \mathbf{W} \big)$ consists of two
elements. To prove Theorem \ref{main: by isotropy and chern}, we
should calculate Chern classes of these two. Pick an element $g_1$
of $G_\chi$ satisfying $\bar{\rho} (g_1) = a_n.$ To calculate Chern
classes, we would describe bundles in $p_\vect^{\prime -1} \big(
\mathbf{W} \big)$ as equivariant clutching construction. For this,
we first express $\RP^2 \cong |\hatK_{\hat{R}}|/Z$ as the quotient
of more simple equivariant simplicial complex. To use results of
\cite{Ki}, we do similar things for $p_\vect^{-1} \Big(
P_1(\mathbf{W}) \Big)$ and $S^2 \cong |\hatK_{\hat{R}}|.$

\begin{figure}[ht!]
\begin{center}
\begin{pspicture}(-5.5,-3.5)(4.5,3.7)\footnotesize

\psline[linewidth=0.5pt](-3.5,3.5)(-5,1)(-4.5,0.5)(-2.5,0.5)(-2,1)(-3.5,3.5)
\psline[linewidth=0.5pt](-3.5,3.5)(-4.5,0.5)(-2.5,0.5)(-3.5,3.5)
\psline[linewidth=0.5pt,linestyle=dotted](-5,1)(-4.5,1.5)(-2.5,1.5)(-2,1)
\psline[linewidth=0.5pt,linestyle=dotted](-3.5,3.5)(-4.5,1.5)(-2.5,1.5)(-3.5,3.5)

\uput[d](-4.5,0.6){$\bar{v}_0^N$} \uput[d](-2.5,0.6){$\bar{v}_1^N$}
\uput[r](-2.1,1){$\bar{v}_2^N$} \uput[ur](-2.5,1.5){$\bar{v}_3^N$}
\uput[ul](-4.5,1.5){$\bar{v}_4^N$} \uput[l](-4.9,1){$\bar{v}_5^N$}
\uput[d](-3.5,0.6){$\bar{e}_0^N$}
\uput[dr](-2.35,0.85){$\bar{e}_1^N$}
\uput[u](-1.9,1.1){$\bar{e}_2^N$} \uput[u](-3.5,1.4){$\bar{e}_3^N$}
\uput[u](-5.1, 1.1){$\bar{e}_4^N$}
\uput[dl](-4.65,0.85){$\bar{e}_5^N$} \uput[u](-3.5,3.5){$N$}

\uput[dl](-2,3.5){$\lineK_{2n}^N$}

\psline[linewidth=0.5pt](-5,-1)(-4.5,-1.5)(-2.5,-1.5)(-2,-1)(-2.5,-0.5)(-4.5,-0.5)(-5,-1)
\psline[linewidth=0.5pt](-3.5,-3.5)(-5,-1)
\psline[linewidth=0.5pt](-3.5,-3.5)(-4.5,-1.5)
\psline[linewidth=0.5pt](-3.5,-3.5)(-2.5,-1.5)
\psline[linewidth=0.5pt](-3.5,-3.5)(-2,-1)

\psline[linewidth=0.5pt](-4.5,-0.5)(-4.16666, -1.5)
\psline[linewidth=0.5pt](-2.5,-0.5)(-2.83333, -1.5)
\psline[linewidth=0.5pt,linestyle=dotted](-3.5,-3.5)(-4.16666, -1.5)
\psline[linewidth=0.5pt,linestyle=dotted](-3.5,-3.5)(-2.83333, -1.5)

\uput[u](-4.5,-0.6){$\bar{v}_4^S$}
\uput[u](-2.5,-0.6){$\bar{v}_3^S$} \uput[r](-2.1,-1){$\bar{v}_2^S$}
\uput[dr](-2.5,-1.5){$\bar{v}_1^S$}
\uput[dl](-4.5,-1.5){$\bar{v}_0^S$} \uput[l](-4.9,-1){$\bar{v}_5^S$}
\uput[u](-3.5,-0.6){$\bar{e}_3^S$}
\uput[ur](-2.35,-0.85){$\bar{e}_2^S$}
\uput[d](-1.9,-1.1){$\bar{e}_1^S$}
\uput[d](-3.5,-1.4){$\bar{e}_0^S$}
\uput[d](-5.1,-1.1){$\bar{e}_5^S$}
\uput[ul](-4.65,-0.85){$\bar{e}_4^S$} \uput[d](-3.5,-3.5){$S$}
\uput[dl](-2,-2.5){$\lineK_{2n}^S$}

\psline[linearc=1]{->}(-0.8, -0.5)(0.8, -0.5) \uput[u](0,
-0.5){$\pi$}

\psline[linewidth=0.5pt](3.5,2.5)(2,0)(2.5,-0.5)(4.5,-0.5)(5,0)(3.5,2.5)
\psline[linewidth=0.5pt](3.5,2.5)(2.5,-0.5)(4.5,-0.5)(3.5,2.5)
\psline[linewidth=0.5pt,linestyle=dotted](2,0)(2.5,0.5)(4.5,0.5)(5,0)
\psline[linewidth=0.5pt,linestyle=dotted](3.5,2.5)(2.5,0.5)(4.5,0.5)(3.5,2.5)

\uput[ur](4.5,0.5){$v_3^N$} \uput[ul](2.5,0.5){$v_4^N$}
\uput[u](5.1,0.1){$e_2^N$} \uput[u](3.5,0.4){$e_3^N$}
\uput[u](1.9,0.1){$e_4^N$}

\uput[u](3.5,2.5){$N$}

\psline[linewidth=0.5pt](3.5,-2.5)(2,0)
\psline[linewidth=0.5pt](3.5,-2.5)(2.5,-0.5)
\psline[linewidth=0.5pt](3.5,-2.5)(4.5,-0.5)
\psline[linewidth=0.5pt](3.5,-2.5)(5,0)

\psline[linewidth=0.5pt,linestyle=dotted](2.5,0.5)(3.5,-2.5)
\psline[linewidth=0.5pt,linestyle=dotted](4.5,0.5)(3.5,-2.5)

\uput[r](4.9,0){$v_2^S$} \uput[dr](4.5,-0.5){$v_1^S$}
\uput[dl](2.5,-0.5){$v_0^S$} \uput[l](2.1,0){$v_5^S$}
\uput[d](5.1,-0.1){$e_1^S$} \uput[d](3.5,-0.4){$e_0^S$}
\uput[d](1.9,-0.1){$e_5^S$} \uput[d](3.5,-2.5){$S$}
\uput[dl](5,-1.5){$\complexK_{2n}$}

\end{pspicture}
\end{center}
\caption{\label{figure: K_6} $\pi : \lineK_{2n} \rightarrow
\complexK_{2n}$ when $n=3$}
\end{figure}
In Table \ref{table: introduction}, $\hatK_{\hat{R}} =
\complexK_{2n}.$ Denote by $\complexK_{2n}^S$ the lower simplicial
subcomplex of $\complexK_{2n}$ such that
\begin{equation*}
|\complexK_{2n}^S| ~ = ~ |\complexK_{2n}| ~ \cap ~ \Big\{ ~ (x,y,z)
\in \R^3 ~ \big| ~ z \le 0 ~ \Big\},
\end{equation*}
and by $\complexK_{2n}^N$ the upper part. Denote by $B$ the set $\{
S, N \},$ and by $\lineK_{2n}$ the disjoint union $\amalg_{q \in B}
~ \complexK_{2n}^q.$ Denote $\complexK_{2n}^q \subset \lineK_{2n}$
by $\lineK_{2n}^q.$ The $(G_\chi \times Z)$-action and its
restricted $G_\chi$-action on $\complexK_{2n}$ induce $(G_\chi
\times Z)$- and $G_\chi$-actions on $\lineK_{2n}$ and
$|\lineK_{2n}|$ where the $G_\chi$-action on $\lineK_{2n}$ preserves
$\lineK_{2n}^q.$ Then, $\complexK_{2n}$ can be regarded as a
quotient of $\lineK_{2n}.$ This is expressed by the simplicial map
$\pi : \lineK_{2n} \rightarrow \complexK_{2n}$ such that $\pi
|_{\complexK_{2n}^q}$ is the identity to $\complexK_{2n}^q$ for $q
\in B.$ Denote by $|\pi| : |\lineK_{2n}| \rightarrow
|\complexK_{2n}|$ the underlying space map of $\pi.$ Easily, $\pi$
and $|\pi|$ are $(G_\chi \times Z)$-maps. Here, we introduce the
following notations:
\begin{align*}
\bar{P}_{2n}^q       &= |\lineK_{2n}^q| \cap |\pi|^{-1} ( P_{2n} ),  \\
\bar{P}_{2n}   &= \bar{P}_{2n}^S \cup \bar{P}_{2n}^N,  \\
\bar{v}_q^i          &= \lineK_{2n}^q \cap \pi^{-1} ( v^i ),  \\
\bar{e}_q^i          &= \lineK_{2n}^q \cap \pi^{-1} ( e^i )
\end{align*}
for $q \in B$ and $i \in I.$ Also, let $c : \bar{P}_{2n} \rightarrow
\bar{P}_{2n}$ be the map satisfying $c(\bar{x}) \ne \bar{x}$ and
$|\pi| ( \bar{x} ) = |\pi| ( c(\bar{x}) )$ for each $\bar{x} \in
\bar{P}_{2n}.$ It is easy that $c$ is $(G_\chi \times
Z)$-equivariant. In this time, we do the similar thing for
$|\complexK_{2n}|/Z.$
\begin{equation} \label{equation: simplicial complices}
\SelectTips{cm}{} \xymatrix{ |\lineK_{2n}| \ar[r]^-{\bar{o}}
\ar[d]_-{|\pi|} & |\complexK_R| \ar[d]^-{|\pi^\prime|} \\
|\complexK_{2n}|   \ar[r]^-{o}  &  |\complexK_{2n}|/Z }
\end{equation}
Put $\complexK_R = \lineK_{2n}^S.$ Let $|\pi^\prime| : |\complexK_R|
\rightarrow |\complexK_{2n}|/Z$ be the map defined by $\bar{x}
\mapsto \big( o \circ |\pi| \big) (\bar{x}).$ Then, $|\pi^\prime|$
is a cellular $G_\chi$-map, and we consider $|\complexK_{2n}|/Z$ as
the quotient of $|\complexK_R|$ through $|\pi^\prime|.$ Also, let
$\bar{o} : |\lineK_{2n}| \rightarrow |\complexK_R|$ be the map
defined by
\begin{equation*}
\bar{o} (\bar{x}) = \left\{
  \begin{array}{ll}
    \bar{x}        & \hbox{for } \bar{x} \in |\lineK_{2n}^S|, \\
    g_0 \bar{x}    & \hbox{for } \bar{x} \in |\lineK_{2n}^N|.
  \end{array}
\right.
\end{equation*}
So defined maps $|\pi^\prime|$ and $\bar{o}$ satisfy the commutative
diagram (\ref{equation: simplicial complices}).

Now, we describe equivariant vector bundles over
$|\complexK_{2n}|/Z$ as an equivariant clutching construction of an
equivariant vector bundle over $|\complexK_R|,$ and then do the
similar thing for equivariant vector bundle over $|\complexK_{2n}|.$
Put $\mathbf{W} = ( W_{d^{-1}}, W_{d^0}, W_{d^1} ).$ Let $F$ be the
equivariantly trivial $G_\chi$-bundle $|\complexK_R| \times
W_{d^{-1}}$ over $|\complexK_R|.$ Here, note that $|\pi^\prime| (S)
= d^{-1},$ and that both $(G_\chi)_S$ and $(G_\chi)_{d^{-1}}$ are
equal to $G_\chi.$ Then, we can show that
\begin{equation}
\label{equation: trivialization 1} F \cong |\pi^\prime|^* E^\prime
\qquad \text{ for each } E^\prime \in p_\vect^{\prime
-1}(\mathbf{W})
\end{equation}
because $\Big( |\pi^\prime|^* E^\prime \Big)_S \cong
E_{d^{-1}}^\prime \cong W_{d^{-1}}$ and $|\complexK_R|$ is
equivariant homotopically equivalent to the one point set $\{ S \}
\subset |\complexK_R|.$ Let $\Phi : \bar{P}_{2n}^S \rightarrow
\Iso(W_{d^{-1}})$ be an arbitrary continuous map which is called a
\textit{preclutching map} with respect to $F$ where $\Iso$ means the
set of nonequivariant isomorphisms. We try to glue $F$ along
$\bar{P}_{2n}^S$
\begin{equation*}
\bar{P}_{2n}^S \times W_{d^{-1}} \longrightarrow \bar{P}_{2n}^S
\times W_{d^{-1}}, \qquad (\bar{x}, u) \mapsto \Big( g_0 c( \bar{x}
), \Phi (\bar{x}) u \Big)
\end{equation*}
for $\bar{x} \in \bar{P}_{2n}^S$ and $u \in W_{d^{-1}}$ via a
preclutching map $\Phi.$ This gluing gives an equivariant vector
bundle in $p_\vect^{\prime -1}(\mathbf{W})$ if and only if the
following two conditions hold:
\begin{enumerate}
  \item[E1$^\prime$.] $\Phi \big( g_0 c( \bar{x} ) \big) = \Phi (\bar{x})^{-1}$
  for each $\bar{x} \in \bar{P}_{2n}^S,$
  \item[E2$^\prime$.] $\Phi (g \bar{x}) = g \Phi ( \bar{x}) g^{-1} ~$ for each
  $g \in G_\chi$ and $\bar{x} \in \bar{P}_{2n}^S.$
\end{enumerate}
A preclutching map satisfying these two conditions is called an
\textit{equivariant clutching map} with respect to $F.$ By
(\ref{equation: trivialization 1}), two elements of $p_\vect^{\prime
-1}(\mathbf{W})$ can be constructed in this way. We do the similar
thing for equivariant vector bundle over $|\complexK_{2n}|.$ Let
$\bar{F}$ be the $(G_\chi \times Z)$-bundle $\bar{o}^* F$ over
$|\lineK_{2n}|.$ Easily, $\bar{F}$ is equal to the equivariantly
trivial bundle $|\lineK_{2n}| \times p_1^* W_{d^{-1}}.$ The $(G_\chi
\times Z)$-action on it is expressed by
\begin{align*}
g \cdot ( \bar{x}, u )    &= ( g \bar{x}, g u ),  \\
g_0 \cdot ( \bar{x}, u )  &= ( g_0 \bar{x}, u )
\end{align*}
for $g \in G_\chi,$ $\bar{x} \in |\complexK_{2n}|,$ $u \in p_1^*
W_{d^{-1}}.$ And, we can show that
\begin{equation}
\label{equation: trivialization 2} \bar{F} \cong |\pi|^* E \qquad
\text{ for each } E \in p_\vect^{-1} \big( P_1(\mathbf{W}) \big)
\end{equation}
because  $|\lineK_{2n}|$ is equivariant homotopically equivalent to
the set $B \subset |\lineK_{2n}|.$ Let $\bar{\Phi} : \bar{P}_{2n}
\rightarrow \Iso(p_1^* W_{d^{-1}})$ be an arbitrary continuous map.
We call it a \textit{preclutching map} with respect to $\bar{F}.$ We
try to glue $\bar{F}$ along $\bar{P}_{2n}$
\begin{equation*}
\bar{P}_{2n} \times p_1^* W_{d^{-1}} \longrightarrow \bar{P}_{2n}
\times p_1^* W_{d^{-1}}, \qquad (\bar{x}, u) \mapsto ( c (\bar{x}),
\bar{\Phi} (\bar{x}) u )
\end{equation*}
for $\bar{x} \in \bar{P}_{2n}$ and $u \in p_1^* W_{d^{-1}}$ via a
preclutching map $\bar{\Phi}.$ This gluing gives an equivariant
vector bundle in $p_\vect^{-1} \big( P_1(\mathbf{W}) \big)$ if and
only if the following two conditions hold:
\begin{enumerate}
  \item[E1.] $\bar{\Phi} (c(\bar{x})) = \bar{\Phi} (\bar{x})^{-1}~$
  for each $\bar{x} \in \bar{P}_{2n},$
  \item[E2.] $\bar{\Phi} (\bar{g} \bar{x}) = \bar{g} \bar{\Phi}
  ( \bar{x}) \bar{g}^{-1}$ for each
  $\bar{g} \in G_\chi \times Z$ and $\bar{x} \in \bar{P}_{2n}.$
\end{enumerate}
A preclutching map satisfying these two conditions is called an
\textit{equivariant clutching map} with respect to $\bar{F}.$  Note
that equivariant clutching maps $\bar{\Phi},$ $\Phi$ with respect to
$\bar{F},$ $F$ are actually $\Iso_H (p_1^* W_{d^{-1}}),$ $\Iso_H
(W_{d^{-1}})$-valued by equivariance, respectively. Here, $\Iso_H
(\cdot)$ is the subgroup of $H$-equivariant elements in $\Iso
(\cdot).$ Also since
\begin{equation*}
\res_H^{G_\chi \times Z} p_1^* W_{d^{-1}} = \res_H^{G_\chi}
W_{d^{-1}},
\end{equation*}
we have $\Iso_H (p_1^* W_{d^{-1}})=\Iso_H (W_{d^{-1}}).$ Details on
equivariant clutching construction can be found in \cite[Section
4]{Ki}.

Let $\Omega_{\bar{F}}$ and $\Omega_F$ be sets of equivariant
clutching maps with respect to $\bar{F}$ and $F,$ respectively. To
use results of \cite{Ki}, we need relate $\Omega_F$ to
$\Omega_{\bar{F}}.$ Consider the map
\begin{equation*}
q_\Omega : \Omega_{\bar{F}} \rightarrow \Omega_F, \quad \bar{\Phi}
\longmapsto \Phi(\bar{x}) = g_0 \bar{\Phi}(\bar{x})
\end{equation*}
for $\bar{x} \in \bar{P}_{2n}^S.$ Then, we have the following:

\begin{lemma} \label{lemma: Omega correspondence}
$q_{\Omega}$ is a well-defined one-to-one correspondence.
\end{lemma}

\begin{proof}
First, we show that $q_\Omega (\bar{\Phi})$ for each $\bar{\Phi}$ is
actually contained in $\Omega_F.$ Put $\Phi = q_\Omega
(\bar{\Phi}).$ For this, we show that $\Phi$ satisfies Condition
E1$^\prime.,$ E2$^\prime.$ Condition E1$^\prime.$ is proved as
follows:
\begin{align*}
 \Phi \big( g_0 c(\bar{x}) \big) &= g_0 \bar{\Phi} \big( g_0 c(\bar{x}) \big)  \\
                                 &= \bar{\Phi} \big( c(\bar{x}) \big) g_0^{-1}  \\
                                 &= \bar{\Phi} ( \bar{x} )^{-1} g_0^{-1}  \\
                                 &= \Big( g_0^{-1} \Phi ( \bar{x} ) \Big)^{-1} g_0^{-1}  \\
                                 &= \Phi ( \bar{x} )^{-1}
\end{align*}
for $\bar{x} \in \bar{P}_{2n}^S.$ And, Condition E2$^\prime.$ is
proved as follows:
\begin{align*}
 \Phi \big( g \bar{x} \big) &= g_0 \bar{\Phi} \big( g \bar{x} \big)  \\
                            &= g_0 g \bar{\Phi} \big( \bar{x} \big) g^{-1}  \\
                            &= g g_0 \bar{\Phi} \big( \bar{x} \big) g^{-1}  \\
                            &= g \Phi \big( \bar{x} \big) g^{-1}
\end{align*}
for each $g \in G_\chi$ and $\bar{x} \in \bar{P}_{2n}^S.$ So,
$q_\Omega (\bar{\Phi})$ is contained in $\Omega_F.$

Second, we construct the inverse of $q_{\Omega}$ to show bijectivity
of it. Consider the map
\begin{equation*}
q_\Omega^{-1} : \Omega_F \longrightarrow \Omega_{\bar{F}}, \quad
\Phi \mapsto \bar{\Phi}(\bar{x}) = \left\{
  \begin{array}{ll}
    g_0^{-1} \Phi (\bar{x})
    & \hbox{for } \bar{x} \in \bar{P}_{2n}^S, \\
    \Big( g_0^{-1} \Phi \big( c(\bar{x}) \big) \Big)^{-1}
    & \hbox{for } \bar{x} \in \bar{P}_{2n}^N.
  \end{array}
\right.
\end{equation*}
We show that $\bar{\Phi} = q_\Omega^{-1} (\Phi)$ for each $\Phi$ is
actually contained in $\Omega_{\bar{F}}.$ For this, we show that
$\Phi$ satisfies Condition E1., E2. Condition E1. is proved by
definition of $\bar{\Phi}.$ We prove Condition E2. only when
$\bar{x} \in \bar{P}_{2n}^S$ and $\bar{g} = g g_0$ for some $g \in
G_\chi$ as follows:
\begin{align*}
\bar{\Phi}( \bar{g} \bar{x}) = \bar{\Phi}(g g_0 \bar{x})
&= \Big( g_0^{-1} \Phi \big( c( g g_0 \bar{x} ) \big) \Big)^{-1}  \\
&= \Phi \big( c( g g_0 \bar{x} ) \big)^{-1} g_0 \\
&= \Phi \big( g_0 c( g \bar{x} ) \big)^{-1} g_0 \\
&= \Phi \big( g \bar{x} \big) g_0 \\
&= g \Phi \big( \bar{x} \big) g^{-1} g_0 \\
&= g g_0 g_0^{-1} \Phi \big( \bar{x} \big) g^{-1} g_0 \\
&= g g_0 \bar{\Phi} \big( \bar{x} \big) g_0^{-1} g^{-1} \\
&= \bar{g} \bar{\Phi} \big( \bar{x} \big) \bar{g}^{-1}
\end{align*}
where we use equivariance of $c.$ Proof of Condition E2. for other
$\bar{x}$'s and $\bar{g}$'s is proved similarly. Since $q_\Omega$
and $q_\Omega^{-1}$ are inverses of each other, we obtain a proof.
\end{proof}

For convenience in calculation, we parameterize each $|e^i|,$
$|\bar{e}_q^i|$ in $|\complexK_{2n}|,$ $|\lineK_{2n}|$ for $0 \le i
\le 2n-1$ by $t \in [i, i+1]$ linearly to satisfy $v_S^i \mapsto i$
and $\bar{v}_S^i \mapsto i.$ Vertices $v_S^0,$ $\bar{v}_S^0$ are
parameterized by 0 or $2n$ according to context. Pick a path $\sigma
: [0,1] \rightarrow \Iso_H (W_{d^{-1}})$ such that
$\sigma(0)=\sigma(1)=\id$ and $[\sigma]$ is a generator of $\pi_1
\Big( \Iso_H (W_{d^{-1}}), \id \Big).$ For a path $\gamma$ defined
on $[0,1],$ define a path $\gamma^{tr}$ on $[0,1]$ as $\gamma^{tr}
(t) = \gamma (1-t)$ for $t \in [0,1],$ and a path $\gamma_{-i}$ on
$[i, i+1]$ as $\gamma_{-i} (t) = \gamma (t-i)$ for $i \in \Z.$ Now,
we prove Theorem B and Corollary \ref{corollary: reduction to line
bundle}.

\begin{proof}[Proof of Theorem B]
We obtain the first and second statement through the isomorphism
(\ref{equation: isomorphism RP^2 and S^2}) because the preimage
$p_{\vect}^{-1} \Big( P_1 (\mathbf{W}) \Big)$ contained in
$\Vect_{G_\chi \times Z} (S^2, \chi)$ has two elements by
\cite[Theorem B]{Ki}. In \cite[Proof of Theorem B]{Ki}, these two
are described as equivariant clutching constructions of $\bar{F}$
via equivariant clutching maps $\bar{\Phi},$ $\bar{\Phi}^\prime$
with respect to $\bar{F}$ satisfying the following:
\begin{enumerate}
  \item $\bar{\Phi} \big|_{\bar{P}_{2n}^S} \equiv \id,$
  \item $\bar{\Phi}^\prime \big|_{\bar{P}_{2n}^S} (t) =
  \left\{
    \begin{array}{ll}
      \sigma (t)             \qquad & \hbox{ for } t \in [0, 1], \\
      (\sigma^{tr})_{-1} (t) \qquad & \hbox{ for } t \in [1, 2].
    \end{array}
  \right.$
\end{enumerate}
Since $\bar{\Phi}^\prime \big|_{\bar{P}_{2n}^S} (t) = \sigma \vee
(\sigma^{tr})_{-1} ~ (t)$ for $t \in [0, 2]$ and $\bar{\Phi}^\prime$
is equivariant, we also have
\begin{equation*}
\bar{\Phi}^\prime \big|_{\bar{P}_{2n}^S} (t) =  g_1^i \Big( ~ \sigma
\vee (\sigma^{tr})_{-1} ~ (t-2i) \Big) g_1^{-i}
\end{equation*}
for an integer $0 \le i \le n-1$ and $t \in [2i, 2i+2]$ where
$\bar{\rho} (g_1) = a_n.$ Put $\Phi = q_\Omega (\bar{\Phi})$ and
$\Phi^\prime = q_\Omega (\bar{\Phi}^\prime).$ By Lemma \ref{lemma:
Omega correspondence}, they are equivariant clutching maps with
respect to $F.$ Denote by $E$ and $E^\prime$ the equivariant vector
bundles determined by $\Phi$ and $\Phi^\prime,$ respectively. It is
easy that the $E$ is equivariantly trivial. Since $\Phi^\prime$ is
equivariant, it is determined by its values on the half $[0, n]$ of
$\bar{P}_{2n}^S.$ Observe that $\Phi^\prime$ on $[1, n] \cup [n+1,
2n]$ does not contribute to $c_1 (E^\prime)$ because $n$ is odd and
the pair $g_1^i \sigma g_1^{-i}$ and $g_1^{i^\prime}
(\sigma^{tr})_{-1} g_1^{-i^\prime}$ for $0 \le i, i^\prime \le n-1$
cancel each other in the fundamental group $\pi_1 \Big( \Iso_H
(W_{d^{-1}}), \id \Big).$ In other words, $E^\prime$ is
nonequivariantly isomorphic to the nonequivariant vector bundle
$E^{\prime \prime}$ determined by the preclutching map $\Phi^{\prime
\prime}$ with respect to $F$ defined by
\begin{equation*}
\begin{array}{ll}
\Phi^{\prime \prime} (t) = \id              & \text{ for } t \in [1,
n] \cup
[n+1, 2n], \\
\Phi^{\prime \prime} (t) = \sigma(t)        & \text{ for } t \in [0,
1],
\\
\Phi^{\prime \prime} (t) = \sigma(t-n)^{-1} & \text{ for } t \in [n,
n+1].
\end{array}
\end{equation*}
Then, we can show that
\begin{equation*}
c_1(E^{\prime \prime}) \equiv \chi(\id) \mod 2
\end{equation*}
in $H^2 \big( \RP^2, \Z \big)$ by \cite[Lemma 7.1]{Ki} because
$[\sigma]$ is a generator of the fundamental group $\pi_1 \Big(
\Iso_H (W_{d^{-1}}), \id \Big).$ Therefore, we obtain a proof.
\end{proof}

\begin{proof}[Proof of Corollary \ref{corollary: reduction to line bundle}]
By \cite[Theorem D]{Ki}, $\Vect_{G_\chi \times Z} (S^2, \chi)$ has a
rank $\chi(\id)$-bundle so that $\Vect_{G_\chi} (\RP^2, \chi)$ has a
rank $\chi(\id)$-bundle by the isomorphism (\ref{equation:
isomorphism RP^2 and S^2}). From this, we have
\begin{equation*}
\Vect_{G_\chi} (\RP^2, \chi) \cong \Vect_R (\RP^2)
\end{equation*}
as semigroups by \cite[Lemma 2.2]{CKMS}. Also, we obtain
\begin{equation*}
\Vect_{\bar{R} \times Z} (S^2) \cong \Vect_R (\RP^2)
\end{equation*}
by the isomorphism (\ref{equation: isomorphism RP^2 and S^2}). It is
easy that $\Vect_R (\RP^2)$ is generated by line bundles and that
$A_R (\RP^2, \id)$ is generated by all the elements with
one-dimensional entries by \cite[Theorem D]{Ki}. Since $R_x=\langle
\id \rangle$ for each $x \in o( P_{2n} ),$ any triple $( W_{d^i}
)_{i \in I^+}$ in $\Rep \Big( R_{d^{-1}} \Big)$ $\times$ $\Rep \Big(
R_{d^0} \Big)$ $\times$ $\Rep \Big( R_{d^1} \Big)$ is contained in
$A_R (\RP^2, \id)$ if and only if $W_{d^i}$'s are same dimensional.
So, the number of all the elements in $A_R (\RP^2, \id)$ with
one-dimensional entries is equal to $n$ because $R_{d^{-1}} \cong
\Z_n$ and $R_{d^0} = R_{d^0} = \langle \id \rangle.$ Therefore, we
obtain a proof.
\end{proof}

\end{document}